\documentclass[11pt,a4paper]{article} 

\usepackage{amsfonts, amsmath, wasysym}
\usepackage{stmaryrd} 

\usepackage{tikz}
\usetikzlibrary{shapes, positioning, patterns}

\usepackage{amssymb,amsthm,
paralist
}

\usepackage{
latexsym,
}

\usepackage{url}

\definecolor{darkgreen}{rgb}{0,0.5,0}
\definecolor{darkred}{rgb}{0.7,0,0}
\usepackage[colorlinks, 
citecolor=darkgreen, linkcolor=darkred
]{hyperref}

\usepackage[a4paper, margin=1in]{geometry}

\theoremstyle{plain}

\numberwithin{equation}{section}

\newcommand{\m}{\ensuremath{{\mathcal M}}}

\newcommand{\g}{\ensuremath{{\mathcal G}}}

\newcommand{\ci}{\ensuremath{{\mathcal I}}}

\newcommand{\calr}{\ensuremath{{\mathcal R}}}

\newcommand{\pl}[2]{{\frac{\partial #1}{\partial #2}}}

\newcommand{\al}{\alpha}
\newcommand{\be}{\beta}
\newcommand{\ga}{\gamma}

\newcommand{\de}{\delta}
\newcommand{\om}{\omega}

\newcommand{\la}{\lambda}

\newcommand{\si}{\sigma}

\newcommand{\ep}{\varepsilon}
\newcommand{\e}{\varepsilon}

\newcommand{\R}{\ensuremath{{\mathbb R}}}
\newcommand{\N}{\ensuremath{{\mathbb N}}}

\newcommand{\C}{\ensuremath{{\mathbb C}}}

\newcommand{\downto}{\downarrow}

\newcommand{\lap}{\Delta}

\newcommand{\grad}{\nabla}

\newcommand{\union}{\cup}

\DeclareMathOperator{\Vol}{Vol}
\DeclareMathOperator{\VolB}{VolB}

\newcommand{\beq}{\begin{equation}}
\newcommand{\beql}[1]{\begin{equation}\label{#1}}
\newcommand{\eeq}{\end{equation}}
\newcommand{\beqa}{\begin{equation}\begin{aligned}}
\newcommand{\eeqa}{\end{aligned}\end{equation}}
\newcommand{\brmk}{\begin{rmk}}
\newcommand{\ermk}{\end{rmk}}
\newcommand{\partref}[1]{\hbox{(\csname @roman\endcsname{\ref{#1}})}}
\newcommand{\half}{\frac{1}{2}}

\newcommand{\Ric}{{\mathrm{Ric}}}
\newcommand{\Scal}{{\mathrm{Scal}}}

\usepackage{soul}

 \newtheorem{thm}{Theorem}[section]

\newtheorem{lem}[thm]{Lemma}

\newtheorem{defn}[thm]{Definition}
\newtheorem{rmk}[thm]{Remark}

\newcommand{\CPICo}{{\ensuremath{\mathrm{C_{PIC1}}}}}
\newcommand{\CPICt}{\ensuremath{\mathrm{C_{PIC2}}}}

\usepackage{varwidth}

\title{\Large PIC1 PINCHED MANIFOLDS ARE FLAT OR COMPACT}
\author{Alix Deruelle, Man-Chun Lee, Felix Schulze, Miles Simon and Peter M. Topping}
\date{20 March 2026}

\begin{document}

%

\parskip 8pt
\parindent 0pt

\maketitle

\begin{abstract}
\noindent Hamilton's pinching conjecture, that three-dimensional complete non-compact manifolds with pinched Ricci curvature are flat, has recently been resolved using Ricci flow. In this paper we prove a direct analogue of that result in all dimensions. In order to do so we develop a lifting technique that allows us to handle manifolds that are collapsed at infinity. This new method also gives an alternative way of handling collapsed manifolds in the known three-dimensional case. As part of this approach, we prove a Ricci flow curvature estimate of a type that would normally be derived from the Harnack inequality, but without requiring the strong curvature positivity hypothesis demanded by Harnack. We give an improved gap theorem as a further application.
\end{abstract}

\section{Introduction}
\label{intro}

One of the central themes in differential geometry is to consider the global topological and geometric consequences of pointwise curvature conditions, typically pointwise curvature bounds or pointwise curvature pinching. 
Since 1982, Ricci flow has been at the heart of this topic. Hamilton's initial application of this flow was to prove that closed 3-manifolds of positive Ricci curvature are diffeomorphic to space forms. Advances in our knowledge of how curvature evolves under Ricci flow, due to B\"ohm and Wilking \cite{BohmWilking2008}, led to a sphere theorem in general dimensions, assuming 
2-nonnegative curvature operator (see \cite{PIC1_survey} for a map of curvature conditions). These ideas were developed 
further by Brendle and Schoen \cite{BS} to give a sphere theorem for positive complex sectional curvature, and hence strictly quarter-pinched sectional curvature.

One would like to find the most general notion of positive curvature that gives a sphere theorem of the same type. The current state of the art is the so-called PIC1 curvature condition. The concept of PIC1
originates in the work of Micallef and Moore \cite{MM} and came to prominence 
in the work of Brendle and Schoen \cite{BS}; 
we give a modern 
definition of this condition in Section \ref{PIC1_sect} and briefly recall some of its properties.
Loosely speaking, PIC1 seems to be the weakest known positive curvature condition that prevents any nontrivial topology; see \cite{PIC1_survey} for some conjectures in this direction.
Building on the theory behind the three sphere theorems mentioned above, Brendle \cite{brendlePIC1} proved the following generalisation of all three of these sphere theorems.
\begin{thm}[Brendle \cite{brendlePIC1}, PIC1 sphere theorem]
\label{PIC1_sphere_theorem}
Suppose $(M,g_0)$ is a closed Riemannian manifold that is PIC1. 
Then $M$ is diffeomorphic to a spherical space form.
\end{thm}
In contrast, this paper considers the more general situation that the manifold is not assumed to be closed, but strengthens the curvature positivity hypothesis to a curvature pinching hypothesis. The expected global consequences are now different, as we illustrate in a moment. However, the frontier curvature condition in this context also turns out to involve the notion of PIC1.

\subsection{Pinching theorems}
\label{pinch_thm_sect}

A conjecture generally attributed to Hamilton \cite[Conjecture 3.39]{CLN09}, but apparently also considered by Willmore in the 1980s \cite{gerhard_personal}, is that Ricci pinched 3-manifolds are either flat or compact. Building on earlier work of Chen-Zhu \cite{ChenZhu2000}, this conjecture was recently solved by a combination of the works of the present authors, together with earlier work of Lott \cite{Lott2019}:
\begin{thm}[{Hamilton's pinching conjecture, \cite{DeruelleSchulzeSimon2025}, \cite{LeeTopping2022} and \cite{Lott2019}}]
\label{3D_thm}
Suppose $(M^3,g_0)$ is a complete three-dimensional Riemannian manifold with $\Ric_{g_0}\geq \ep\, \Scal_{g_0}\geq 0$ for some $\ep>0$. Then  $(M^3,g_0)$ is either flat or  compact.
\end{thm}
Note that the condition $\Ric_{g_0}\geq \ep\, \Scal_{g_0}\geq 0$ is telling us that at each point in $M$, the eigenvalues of the Ricci curvature are all comparable.
After \cite{DeruelleSchulzeSimon2025} appeared, alternative approaches to that part were
given first in \cite{HK} and subsequently in e.g. \cite{BQOP}, \cite{BMOP} and  \cite{ChenXuZhang}.

In this paper we are concerned with establishing a direct analogue of this result in higher dimensions.
By considering the Eguchi-Hanson space, which is a 4-dimensional, complete Ricci flat manifold that is neither flat nor compact, we see that a naive restatement of the same conjecture in higher dimensions is not appropriate (although see \cite[Remark 3.1]{Ni_ancient} for a possible reformulation assuming also strictly positive scalar/Ricci 
curvature).
Instead we notice that non-negative Ricci curvature and Ricci pinching are
equivalent to (weakly) PIC1 curvature and PIC1 pinching, respectively, in three dimensions. 
It is the PIC1 pinching formulation of Theorem \ref{3D_thm} that we will generalise.

We are now in a position to state our main theorem. The $n=3$ case of the theorem below recovers
Theorem \ref{3D_thm}. Those unfamiliar with the notion of curvature cones should consult the review in Section \ref{PIC1_sect}; in particular, $\CPICo$ will be a cone within the vector space of algebraic curvature tensors so that a manifold $(M,g)$ being PIC1 is equivalent to its curvature tensor $\calr_g$ lying in the interior of $\CPICo$. The curvature tensor $\ci$ is the `identity' curvature tensor; see Section \ref{PIC1_sect}.
\begin{thm}[Main theorem]
\label{main_thm}
Suppose for $n\geq 3$ that $(M^n,g)$ is a complete Riemannian manifold that is PIC1 pinched in the sense that there exists $\ep\in (0,\frac{1}{n(n-1)})$ such that 
\beql{main_thm_pinching_hyp}
\calr_{g}-\e\, \Scal_{g} \cdot \ci\in \CPICo.
\eeq
Then $M$ is compact, or $g$ is flat (or both).
\end{thm}
As explained in Remark \ref{pinch_trace} below,
the restriction $\e<\frac{1}{n(n-1)}$, together with the pinching hypothesis \eqref{main_thm_pinching_hyp}, implies that 
$\Scal_{g_0}\geq 0$, which then makes \eqref{main_thm_pinching_hyp} stronger than asking
that $\calr_{g}\in \CPICo$, and consequently guarantees that $\Ric_g\geq 0$.
Theorem \ref{main_thm} was explicitly proposed in \cite[Remark 1.4]{LeeTopping2022_PIC1}, 
and was later listed as Problem (j) in \cite[Section 7]{milessurvey}, 
but is a natural generalisation of several earlier theorems. 
The following results have hypotheses that are a strict superset of those in Theorem \ref{main_thm}.

\begin{thm}[{Brendle-Schoen \cite[Theorem 7.4]{BSsurvey}}]
In the setting of Theorem \ref{main_thm}, the same conclusions follow if we additionally assume:
\begin{compactenum}
\item
$(M,g)$ has the stronger notion of PIC2 pinching, i.e., $\calr_{g}-\e\, \Scal_{g} \cdot \ci$ lies in the curvature cone corresponding to non-negative complex sectional curvature\footnote{See \cite{PIC1_survey} for a discussion of PIC2.}, 
\item 
$(M,g)$ has uniformly bounded sectional curvature, and
\item
$(M,g)$ has strictly positive scalar curvature.
\end{compactenum}
\end{thm}

A by-product of the hypothesis of positive scalar curvature is that the manifold cannot be flat, so the conclusion here is that $M$ is compact.
This result extended earlier work of Ni and Wu \cite{NiWu2007} and Chen-Zhu \cite{ChenZhu2000}.

More recently, two of the present authors obtained:
\begin{thm}[{Lee-Topping \cite{LeeTopping2022_PIC1}}]
\label{PIC1_pinching_thm_LT}
In the setting of Theorem \ref{main_thm}, the same conclusions follow if we additionally assume that $(M,g)$ has nonnegative complex sectional curvature.
%
\end{thm}
The additional hypothesis here simplified part of the argument by permitting the use of Hamilton's Harnack inequality \cite{Ham_RF_Harnack}, in a form due to Brendle \cite{Brendle2009}, in a blow-down argument following Schulze-Simon \cite{SchulzeSimon2013},
cf. Chen-Zhu \cite{ChenZhu2000}.

More recently still, three of the present authors replaced  the hypothesis of 
nonnegative complex sectional curvature by the hypothesis of positive asymptotic volume ratio (see also the later work of Chan-Lee-Peachey 
\cite[Corollary 6.3]{ChanLeePeachey2024}). As a by-product of the positive asymptotic volume ratio, this then implies additionally that the manifold must be Euclidean space:
\begin{thm}[{Deruelle-Schulze-Simon \cite[Theorem 1.3]{DSS2}}]
\label{DSS_AVR_thm}
In the setting of Theorem \ref{main_thm}, 
if we additionally assume that $(M,g)$ satisfies 
$$\mathrm{AVR}(g):=\lim_{r\to\infty} \frac{\VolB_g(x_0,r)}{\omega_n r^n}>0,$$
where $\omega_n$ is the volume of Euclidean unit ball in $\R^n$,
then $(M,g)$ is isometric to Euclidean space.
\end{thm}
The point $x_0\in M$ is arbitrary in the theorem above; it is well known that 
$\mathrm{AVR}(g)$ is well-defined independently of the choice of $x_0$ because 
$\Ric_g\geq 0$ as a consequence of the PIC1 pinching hypothesis (see Section \ref{PIC1_sect}). We use the notation $\VolB_g(x_0,r)$ to refer to the volume of 
the ball $B_g(x_0,r)$ in $(M,g)$, with respect to the volume measure of $g$.

The proof of our main theorem \ref{main_thm} will directly appeal to Theorem \ref{DSS_AVR_thm}. It will also appeal to the key ingredient of the proof of 
Theorem \ref{PIC1_pinching_thm_LT}, and its analogue in the proof of 
Theorem \ref{3D_thm}, which is the following Ricci flow existence result.

Below, we use the terminology $|K|_{g(t)}$ to refer to the function on $(M,g(t))$ that gives, at $x\in M$, the magnitude of the largest sectional curvature at $x$. It is equivalent to $|\calr|_{g(t)}$ but only up to a constant depending on $n$.
\begin{thm}[Lee-Topping \cite{LeeTopping2022} and \cite{LeeTopping2022_PIC1}]
\label{Thm:existence}
For any $n\geq 3$ and $\e\in (0,\frac{1}{n(n-1)})$, there exist $c_0>0$ and
$\e_0\in (0,\frac{1}{n(n-1)})$ such that 
the following holds. Suppose $(M^n,g_0)$ is a complete non-compact manifold 
such that
\begin{equation}
\label{e0_pinching_hyp}
\calr_{g_0}-\e\, \Scal_{g_0} \cdot \ci\in \CPICo
\end{equation}
on $M$. Then there exists a smooth complete Ricci flow solution $g(t)$ on $M$ for 
$t\in [0,\infty)$ such that $g(0)=g_0$ and 
\begin{enumerate}[(a)]
\item
\label{LT_exist_a}
$\calr_{g(t)}-\e_0\, \Scal_{g(t)} \cdot \ci\in \CPICo$; 
\item 
\label{LT_exist_b}
$|K|_{g(t)}\leq c_0 t^{-1}$.
\end{enumerate}
for all $t>0$.
\end{thm}

\begin{rmk}
The cone $\CPICo$ is larger than the cone of algebraic curvature operators with  
2-nonnegative curvature operator (see \cite{PIC1_survey}) so our main theorem \ref{main_thm} also resolves the 2-nonnegative pinching conjecture in \cite[Question 1.6]{DeruelleSchulzeSimon2025} as a by-product.
\end{rmk}


An important feature of Theorem \ref{Thm:existence} is that no initial assumptions on volume are  necessary, and initial boundedness of curvature is not assumed.  Hence,  the theorem  may be used to obtain an immortal  solution satisfying \eqref{LT_exist_a} and 
\eqref{LT_exist_b}   starting at the metric given in the statement of Theorem \ref{main_thm}.

As mentioned above, the reason for the non-negative complex sectional assumption in Theorem \ref{PIC1_pinching_thm_LT} is that it gives access to the Harnack inequality,
which can be used to show that certain parabolic blow-downs are expanding solitons \cite[Conjecture 16.6]{formations}, \cite[Theorem 4.3]{ChenZhu2000}, 
\cite[Theorem 1.2]{SchulzeSimon2013}.
One  consequence of the Harnack inequality is that for each point $x_0\in M$, the function $t\mapsto t\,\Scal_{g(t)}(x_0)$ is increasing.\footnote{In this paper \emph{increasing} refers to weakly increasing rather than strictly increasing.}
In particular, knowing that we have a single point $x_0\in M$ at which $\Scal_{g(0)}(x_0)>0$ (and hence have positivity of the scalar curvature throughout for positive time, by the maximum principle) implies that 
$t\,\Scal_{g(t)}(x_0)\geq \Scal_{g(1)}(x_0)>0$ for all $t\geq 1$. This then gives us  the   asymptotic control 
$$\liminf_{t\to\infty} \ t\,\Scal_{g(t)}(x_0)>0,$$
which guarantees that  parabolic blow-downs  are non-trivial. 
The following theorem gives us this same control for immortal solutions 
with $c_0/t$ curvature decay, assuming  the   weaker  condition of non-negative Ricci curvature in place of non-negative complex sectional curvature, and is a crucial ingredient in the proof of our main theorem \ref{main_thm}.

\begin{thm}
\label{lifted_flow_app_intro}
Suppose $n\geq 3$ and $(M^n,g(t))$ is a complete  
Ricci flow for $t\in [0,\infty)$, and  $x_0\in M$. Suppose that there exists $c_0>0$  so that 
\begin{enumerate}[(A)]
\item 
\label{harnack_type_part_a}
$\Ric_{g(t)}\geq 0$  for all $t\in [0,\infty)$; 
\item 
\label{harnack_type_part_b}
$|K|_{g(t)}\leq c_0 t^{-1}$ for all $t\in (0,\infty)$;
\item 
$\Ric_{g(0)}(x_0)>0$.
\end{enumerate}
Then
\beql{Ric_pos_conc_intro}
\liminf_{t\to\infty} \ t\,\Scal_{g(t)}(x_0)>0.
\eeq
\end{thm}


\subsection{The gap phenomenon for manifolds with weakly PIC1}



The gap phenomenon is another situation in which a complete non-compact manifold with some form of non-negative curvature can be seen to be flat by virtue of an additional curvature hypothesis.
Instead of a pinching hypothesis, gap theorems assume sufficiently fast curvature decay at spatial infinity in a pointwise or averaged sense; see, for example,
\cite{MokSiuYau,GreeneWu,EschenburgSchroederStrake,Dress}. In the K\"ahler case, Ni \cite{Ni} showed that a complete non-compact K\"ahler manifold $(M,g)$ with non-negative bisectional curvature that satisfies the curvature decay
\begin{equation}
   \liminf_{r\to+\infty} \frac{r^2}{\VolB_{g}(x_0,r)} \int_{B_{g}(x_0,r)} \mathrm{Scal_g}\,d\mathrm{vol}_{g}=0
\end{equation}
for some $x_0\in M$, must be flat.
While the optimal Riemannian analogue remains unclear, Chan-Lee \cite{ChanLee2025} showed that in the case of Euclidean volume growth, an analogue of Ni's gap theorem holds under the 
weakly PIC1 curvature condition. In the non-Euclidean volume growth case, it  was further shown that   there exists a dimensional constant $\e_n>0$ such that if a complete non-compact manifold $(M,g)$ with non-negative complex sectional curvature  satisfies 
\begin{equation}
          \sup_{x\in M} \int^\infty_0 \left(\frac{r}{\VolB_{g}(x,r)} \int_{B_{g}(x,r)} \mathrm{Scal_{g}}\,d\mathrm{vol}_{g}\right)\,dr <\e_n,
      \end{equation}
then it is necessarily flat.  The argument in \cite{ChanLee2025}  used to show this fact relied   on Brendle's Harnack inequality \cite{Brendle2009} for manifolds with  non-negative complex sectional curvature. 
In order to illustrate the applicability of the new methods in this paper, we will use 
Theorem~\ref{lifted_flow_app_intro} in place of Brendle's Harnack inequality 
to improve this result to the setting that the manifold is weakly PIC1, which is weaker than the setting of non-negative complex sectional  curvature (see e.g.~\cite{PIC1_survey}).
\begin{thm}\label{thm:gap}
For all $n\geq 4$, 
there exists $\e_n>0$ such that the following holds: Suppose $(M^n,g_0)$ is a complete non-compact manifold with
\begin{itemize}
      \item[(1)] $\mathcal{R}_{g_0}  \in  \CPICo$;
      \item[(2)] 
      the sectional curvature satisfies $K_{g_0}\geq -1$;
      \item[(3)] for all $x\in M$, 
      \begin{equation}
          \int^\infty_0 \left(\frac{r}{\VolB_{g_0}(x,r)} \int_{B_{g_0}(x,r)} \mathrm{Scal_{g_0}}\,d\mathrm{vol}_{g_0}\right)\,dr <\e_n.
      \end{equation}
\end{itemize}
Then $(M,g_0)$ is flat.
\end{thm}

\subsection{Outline of the proof of Theorem \ref{main_thm}}

To understand the innovations in this paper, we need to understand the existing overarching strategy to prove pinching theorems with Ricci flow that originates in the work of 
Chen-Zhu \cite{ChenZhu2000}. Suppose for a contradiction that we have a non-compact non-flat pinched manifold $(M,g)$. First, one hopes to be able to run the Ricci flow for all time starting with $(M,g)$, so that the flow has estimates including $|K|\leq C/t$ curvature decay. Traditionally that has been impossible without imposing additional curvature hypotheses, but in this paper the flow will be provided by Theorem \ref{Thm:existence}.
Next, one does a parabolic blow-down of the Ricci flow and hopes to extract a nonflat limit that one can think of as a parabolic tangent cone at infinity. 
Morally this limit should be a non-trivial expanding Ricci soliton that is also pinched,
and one works towards a contradiction guided by the idea that no such flow should exist.

In this paper, this rough idea would be realised precisely by Theorem \ref{DSS_AVR_thm}, but for the hypothesis of positive asymptotic volume ratio. Without that condition $\mathrm{AVR}>0$, the parabolic blow-down fails since the whole flow would collapse. Up until this paper the only situation in which this issue could be addressed without additional hypotheses was in three dimensions, using the work of Lott \cite{Lott2019}.
In this paper we find a method for handling the possibility of zero 
asymptotic volume ratio that works in all dimensions; as a by-product this
gives a completely different approach to that of Lott in three dimensions.

The new method, in its simplest form, is heuristically as follows. First we advance through the Ricci flow until a fixed later time $t=t_0$ so that the curvature is bounded. Next we use the exponential map at that time to take a local cover of the flow at a point $x_0\in M$, always working within the conjugate radius so that the exponential map is a local diffeomorphism. This map, which is fixed independent of time, is then used to pull back the earlier part of the Ricci flow. 
The resulting Ricci flow has traded completeness for non-collapsing. 

This whole approach is applied not just to the original Ricci flow $(M,g(t))$, but also to its parabolic blow-downs. The result is a sequence of local non-collapsed Ricci flows
that subconverges to a local non-collapsed Ricci flow for positive time that attains rough initial data in a weak sense that  has a nontrivial tangent cone at $x_0$.
We can then blow up this blown-down Ricci flow to get a \emph{complete} Ricci flow for all $t>0$ that attains a \emph{non-collapsed} cone as initial data. All the curvature estimates, including pinching, pass to this blow-up  $(\m,H_\infty(t))$. Crucially, this flow now has positive asymptotic volume ratio, but is non-flat. It is \emph{this} Ricci flow that can be used to get a contradiction via Theorem \ref{DSS_AVR_thm}.

In practice, we don't quite phrase the proof like this.
As a first deviation from these heuristics, in order to make the method more applicable for other problems, we split the argument into two: The non-flatness of the initial data $g$ will be shown to imply that the scalar curvature of $g(t)$ 
cannot decay faster than a rate $c/t$, while the pinching condition will be shown to imply that the scalar curvature of $g(t)$ 
\emph{has} to decay 
faster than a rate $c/t$, 
giving a contradiction. 
The 
first part here is essentially the content of Theorem \ref{lifted_flow_app_intro}. Both parts will be covered in Theorem \ref{lifted_flow_app}.

As a second deviation from these heuristics, we  opt for an equivalent method that avoids considering the blown-down initial data (only the blown-down flow) and so avoids all local metric geometry and the associated technical complications. Instead, we pass basic information about the volumes of balls in the approximating flows from $t=0$ to a positive time, and then reason that these volume bounds pass to the limit flow because of the smooth local convergence of flows for $t>0$. 

The precise proof is assembled in Section \ref{main_thm_pf}.
Before that we give a definition and properties of PIC1 in Section \ref{PIC1_sect} and
discuss ball inclusion lemmas in Section \ref{ball_inclusion_sect}, including a new \emph{eternally shrinking balls lemma \ref{eternally_shrinking_balls_lem}}.
We show how positive Ricci curvature in the initial data of a Ricci flow can lead to upper bounds on the volume of balls for a definite time in Section \ref{Ricci_ball_sect},
and construct lifted Ricci flows with uniform lower bounds on the volume of balls
in Section \ref{lifted_flow_sect}.
The proof of Theorem \ref{lifted_flow_app} is given in Section \ref{harnack_sub_sect}.

\bigskip

\noindent
\emph{Acknowledgements:} 
For the purpose of open access, the authors have applied a Creative Commons Attribution (CC BY) licence to any author accepted manuscript version arising. A.~Deruelle is partially supported by grants from the French National Research Agency
ANR-24-CE40-0702 (Project OrbiScaR) and the Charles Defforey Fondation-Institut de France via
the project “KRIS”. He also benefits from a Junior Chair from the Institut Universitaire de France. M.-C.~Lee is supported by Hong Kong RGC grants No.~14300623 and No.~14304225, and an Asian Young Scientist Fellowship. F.~Schulze has received funding from the European Research Council (ERC) under the European Union’s Horizon 1.1 research and innovation programme, grant agreement No.~101200301
(GENREG). M. Simon  was supported by the Special Priority Program SPP 2026    “Geometry at Infinity” of the German Research Foundation (DFG).

\section{The PIC1 curvature condition}
\label{PIC1_sect}

In this section we survey the PIC1 notion of positive curvature, summarising the exposition in \cite{PIC1_survey}. PIC1 is a condition that is the same as positive Ricci curvature in three dimensions, but is a bit stronger in higher dimensions. 
Its relationship with other familiar  notions of positive curvature is summarised in 
\cite[Figure 1]{PIC1_survey}.
It is possible to understand the proof of the main theorem \ref{main_thm} largely using this intuition alone provided one is willing to accept
the existence theory of Theorem \ref{Thm:existence} and the positive asymptotic volume ratio pinching result in Theorem \ref{DSS_AVR_thm}.


One normally discusses notions of non-negativity of curvature in terms of 
\emph{curvature cones}.
Working in $\R^n$, we define the vector space of algebraic curvature tensors
$\mathcal{C}_B(\R^n)$ to be the symmetric bilinear forms on 
$\Lambda^2\R^n$ that satisfy the Bianchi identity. 
The element $\ci\in \mathcal{C}_B(\R^n)$ is the natural extension of the standard inner product on $\R^n$ to $\Lambda^2\R^n$, normalised so that $\ci(e_1\wedge e_2,e_1\wedge e_2)=1$ for all orthogonal unit vectors $e_1,e_2\in \R^n$.


A curvature cone is a closed, convex, $O(n)$-invariant cone within $\mathcal{C}_B(\R^n)$.
Using the $O(n)$-invariance, we can make an isometric identification of any tangent space of a $n$-dimensional Riemannian manifold with $\R^n$ and view the curvature tensor at that point as an algebraic curvature tensor. 
For example, $\ci$ corresponds to the curvature tensor of the unit $n$-sphere.
The manifold is then said to satisfy the  curvature condition corresponding to the cone 
if the 
curvature tensor at each point lies in this cone (or its interior, depending on whether we ask for positive or non-negative curvature). 
For example, the cone of algebraic curvature operators having non-negative inner product with $\ci$ would correspond to non-negative scalar curvature.
Indeed, we can write the scalar curvature as $\langle \calr,\ci\rangle$ up to a 
normalisation factor. 


Although we  only discuss (real) Riemannian manifolds in this paper, many of the most natural curvature cones are defined using the algebra of complexification \cite{MM}.
We can extend any $\calr\in \mathcal{C}_B(\R^n)$ by complex linearity to a symmetric bilinear form on $\Lambda^2\C^n$.
In \cite{PIC1_survey} the following new, self-contained  definition of the PIC1 curvature cone $\CPICo$ was given:
$$\CPICo:=\{\calr\in \mathcal{C}_B(\R^n)\ :\ \calr(\om,\overline{\om})\geq 0
\text{ for all simple }\om\in \Lambda^2\C^n \text{ with }\ci(\om,\om)=0\}.$$
The original definition from \cite{MM} was in terms of the notion of positive isotropic curvature, which we do not need in this paper.
In the sequel we will use implicitly that $\ci\in\CPICo$.

Our ability to prove theorems about PIC1 or PIC1 pinched manifolds was transformed by the eventual realisation that the PIC1 condition was preserved under Ricci flow (e.g. on closed manifolds) \cite{hamposisocurv, N10, BS, wilking2013}, see also \cite[Section 3]{PIC1_survey}. In this paper that property is hidden in the existence theory of Theorem \ref{Thm:existence}.

\begin{rmk}
\label{pinch_trace}
As mentioned above, if $\calr\in \CPICo$, then the corresponding Ricci and hence scalar curvatures are non-negative, i.e. $\langle \calr,\ci\rangle \geq 0$.
In particular, if the pinching  condition $\calr_{g}-\e\, \Scal_{g} \cdot \ci\in \CPICo$ holds then $ \langle \calr_{g}-\e\, \Scal_{g} \cdot \ci , \ci\rangle\geq 0$.
Keeping in mind that the scalar curvature of $\ci$ is $n(n-1)$, we find that 
$$(1 - \ep n(n-1))\Scal_g\geq 0.$$
In particular, if $\ep<\frac{1}{n(n-1)}$ then $\Scal_g\geq 0$, 
which in turn makes the pinching condition imply that $\calr_g\in \CPICo$,
and thus $\Ric_g\geq 0$ (see, e.g. \cite{PIC1_survey}).
\end{rmk}

\section{Shrinking balls lemmas}
\label{ball_inclusion_sect}

In \cite{ST1}, Hamilton's observations about how fast distances can decrease under Ricci flow were adapted to the setting of local Ricci flow in the following lemma.

\begin{lem}[{The shrinking balls lemma, \cite[Corollary 3.3]{ST1}}]
\label{shrinking_balls_lem}
Suppose $(M,g(t))$ is a Ricci flow for $t\in [0,T]$ on a manifold $M$ of any dimension $n$.
Then there exists $\be=\be(n)\geq 1$ so that the following is true.
Suppose that $x_0\in M$ and  $B_{g(0)}(x_0,r)\Subset M$ 
for some $r>0$,  and, for some $c_0>0$, that $\Ric_{g(t)}\leq (n-1)c_0/t$,
on $B_{g(0)}(x_0,r)$ for all $t\in (0,T]$.
Then 
\beq
B_{g(0)}(x_0,r)\supset B_{g(t)}\left(\textstyle{x_0,r-\be\sqrt{c_0 t}}\right)
\eeq
for all $t\in [0,T]$.
\end{lem}
Unfortunately, this lemma becomes vacuous once $t=\frac{r^2}{c_0 \be^2}$. 
We will need an alternative ball inclusion lemma that works for arbitrarily large times. The concept and proof of the following result is close to the bi-H\"older distance estimates of 
\cite[Lemma 3.1]{ST2}.

\begin{lem}[The eternally shrinking balls lemma]
\label{eternally_shrinking_balls_lem}
Suppose $(M,g(t))$ is a Ricci flow (not necessarily complete) for $t\in [0,T]$ on a manifold $M$ of any dimension $n$.
Suppose $x_0\in M$, $R>0$ and $c_0>0$ such that 
\begin{compactenum}
\item
$B_{g(0)}(x_0,R)\Subset M$ and 
\item
$\Ric_{g(t)}\leq (n-1)c_0/t$ on $B_{g(0)}(x_0,R)$ for all $t\in (0,T]$.
\end{compactenum}
Then there exists $r>0$ depending only on $n$, $c_0$, $R$ and $T$ (or merely an upper bound for $T$) such that for every $t\in [0,T]$ we have 
\beq
\label{eternal_SBL_conc}
B_{g(0)}(x_0,R)\supset B_{g(t)}(x_0,r).
\eeq
\end{lem}

\begin{proof}
Define $t_0:=\frac{R^2}{4\be^2 c_0}$.
Suppose first that $t\in [0,T]$ satisfies $t\leq t_0$, i.e.
$\be \sqrt{c_0 t}\leq \frac{R}{2}$.
Then the shrinking balls lemma \ref{shrinking_balls_lem} tells us that 
$$\textstyle B_{g(0)}(x_0,R)\supset B_{g(t)}(x_0,R-\be \sqrt{c_0 t})\supset B_{g(t)}(x_0,\frac{R}{2}),$$
so we have proved \eqref{eternal_SBL_conc} in the case $t\in [0,t_0]$ provided we insist that $r\leq \frac{R}{2}$.

Suppose instead that $t>t_0$. The previous argument applied with $t=t_0$ tells us that
\beql{up_to_t0}
\textstyle B_{g(0)}(x_0,R)\supset B_{g(t_0)}(x_0,\frac{R}{2}).
\eeq
For $s\in [t_0,T]$,  we have $\Ric_{g(s)}\leq (n-1)c_0/t_0$ on $B_{g(0)}(x_0,R)$,
so for any tangent vector $X$ at a point in $B_{g(0)}(x_0,R)$ we have
$$\frac{d}{ds}\log \left[g(s)(X,X)\right] = \frac{-2\Ric_{g(s)}(X,X)}{g(s)(X,X)}\geq -2(n-1)c_0/t_0.$$
Integrating from $s=t_0$ to $s=t$ gives 
$$g(t)(X,X)\geq g(t_0)(X,X) e^{\frac{-2(n-1)c_0}{t_0}(t-t_0)},$$
so 
$$\sqrt{g(t)(X,X)}\geq \al \sqrt{g(t_0)(X,X)} \qquad
\text{ where }\al:=e^{\frac{-(n-1)c_0}{t_0}T}\in (0,1].$$
This implies that the length of a $C^1$ curve within $B_{g(0)}(x_0,R)$ cannot shrink by more than a factor $\al$ between times $t_0$ and $t$. 

We claim that this implies that $B_{g(0)}(x_0,R)\supset B_{g(t)}(x_0,\half\al R)$.
Suppose for a contradiction that there exists $x\in B_{g(t)}(x_0,\half\al R)$ that is 
not in $B_{g(0)}(x_0,R)$. Pick a $C^1$ curve $\ga:[0,1]\to M$ with $\ga(0)=x_0$, 
$\ga(1)=x$ and length $L_{g(t)}(\ga)<\half\al R$. Let $a\in (0,1]$ be the infimum of all $\si\in [0,1]$ such that $\ga(\si)\notin B_{g(0)}(x_0,R)$. Then the image of the entire curve 
$\ga:[0,a)\to M$ lies in $B_{g(0)}(x_0,R)$, but $\ga(a)$ does not.
But the estimates above imply 
$$L_{g(t_0)}(\ga|_{[0,a]})\leq \frac{1}{\al}L_{g(t)}(\ga|_{[0,a]})
\leq \frac{1}{\al}L_{g(t)}(\ga)<\frac{R}2,$$ 
so $\ga(a) \in B_{g(t_0)}(x_0,\frac{R}{2})\subset 
B_{g(0)}(x_0,R)$, by \eqref{up_to_t0}, giving a contradiction.
Thus we have proved \eqref{eternal_SBL_conc} also in the case of $t\in (t_0,T]$, 
provided we ask that $r\leq \half\al R$.

We conclude by choosing $r=\min(\frac{R}{2},\half\al R)=\half\al R$.
\end{proof}

\section{{$\eta$-Ricci balls}}
\label{Ricci_ball_sect}


The following notion is a variation of the \emph{curvature bumps} of Hamilton \cite[Definition 21.1]{formations}; we  control the Ricci curvature instead of the sectional curvature.

\begin{defn}
\label{Ricci_ball_def}
Suppose $(N,g)$ is a Riemannian manifold, $x_0\in N$ and $r>0$, with 
$B_g(x_0,r)\Subset N$. For $\eta>0$, we call $B_g(x_0,r)$ an $\eta$-Ricci ball
of radius $r$ if $\Ric_g\geq \frac{\eta}{r^2}$ on $B_g(x_0,r)$.
\end{defn}


This definition is scale invariant in the following sense: If $B_g(x_0,r)$ is an $\eta$-Ricci ball of radius $r$ in $(N,g)$ and $\la>0$ then $B_{\la^2 g}(x_0,\la r)$ is an $\eta$-Ricci ball of radius $\la r$ in $(N,\la^2 g)$.

\begin{lem}
\label{Miles_Ricci_balls_lem}
Suppose $(N^n,g)$ is a Riemannian manifold, $x_0\in N$, $r>0$,  
$B_g(x_0,r)\Subset N$, $\eta>0$ and $B_g(x_0,r)$ is an $\eta$-Ricci ball of 
radius $r$. Then there exists $\de\in (0,\half)$ depending only on $n$ and $\eta$ such that $\VolB_g(x_0,r)\leq (1-2\de)\om_n r^n$, 
where $\om_n$ is the volume of the Euclidean unit ball in $\R^n$. 
\end{lem}

\begin{proof}
Define $s:=\sqrt{\frac{\eta}{n-1}}>0$.
By scale invariance, with respect to the scaled metric $h=\frac{s^2}{r^2} g$
we have $\Ric_h\geq n-1$ on $B_h(x_0,s)$.
By Bishop-Gromov we have 
$$\VolB_h(x_0,s)\leq \VolB_{S^n}(s)   := (1-2\de)\om_n s^n,$$
where we use the shorthand $\VolB_{S^n}(s)$ to represent the volume of a ball 
of radius $s$ in the unit $n$-sphere, and $\de\in (0,\half)$ is defined (depending only on $s$, i.e. only on $n$ and $\eta$) to make the final equality true. 
Using the definition of $h, s$  and this inequality, we see
$$\frac{\VolB_g(x_0,r)}{\om_n r^n}=  \frac{\VolB_h(x_0,s)}{\om_n s^n}\leq 1-2\de,$$
as required.
%
%
\end{proof}

The reason we work with $2\de$ rather than $\de$ is to make the constants cleaner in the following lemma. 

\begin{lem}[Suppressed volume ratio persists for a uniform time]
\label{Ricci_balls_vol_suppression_lem}
Suppose $(N^n,g(t))$ is a Ricci flow for $t\in [0,T]$, not necessarily complete, and $x_0\in N$, $c_0>0$, $\de\in (0,\half)$ and $r>0$  with 
\begin{compactenum}[(i)]
\item
\label{sv_p1}
$B_{g(t)}(x_0,r)\Subset N$ for all $t\in [0,T]$,
\item
\label{sv_p2}
$\VolB_{g(0)}(x_0,r)\leq (1-2\de)\om_n r^n$,
\item
\label{sv_p3}
$\Ric_{g(t)}\leq (n-1)c_0/t$ on $B_{g(0)}(x_0,r)$ for all $t\in (0,T]$.
\item
\label{sv_p4}
$\Ric_{g(t)}\geq 0$ on $B_{g(0)}(x_0,r)$ for all $t\in [0,T]$.
\end{compactenum}
Then there exists $\rho=\rho(n,\de,c_0)>0$ such that 
$$\VolB_{g(t)}(x_0,r)\leq (1-\de)\om_n r^n\quad\text{ for all }t\in [0, \min(\rho r^2,T)].$$
\end{lem}

\begin{proof}
By hypothesis \eqref{sv_p4} that $\Ric\geq 0$, the volume of fixed sets decreases under Ricci flow (see e.g. \cite[(2.5.7)]{RFnotes}).
Therefore by hypothesis \eqref{sv_p2}, we have 
\beql{init_vol_est}
\Vol_{g(t)}(B_{g(0)}(x_0,r))\leq (1-2\de)\om_n r^n\quad\text{ for all }
t\in [0,T].
\eeq
Define $\la:=\left[\frac{1-2\de}{1-\de}\right]^\frac{1}{n}\in (0,1)$. 
By the shrinking balls lemma \ref{shrinking_balls_lem}, by choosing $\rho>0$ sufficiently small, depending only on $n$, $c_0$ and $\de$, we can be sure that for $t\in [0, \min(\rho r^2,T)]$ we have
$B_{g(0)}(x_0,r)\supset B_{g(t)}(x_0,\la r)$.
Therefore, using also \eqref{init_vol_est}, and the definition of $\la$,
we obtain
$$\VolB_{g(t)}(x_0,\la r)\leq \Vol_{g(t)}(B_{g(0)}(x_0,r))\leq (1-2\de)\om_n r^n
\leq \frac{(1-2\de)}{\la^n}\om_n (\la r)^n=
(1-\de)\om_n (\la r)^n$$
for all  $t\in [0, \min(\rho r^2,T)]$.
By hypothesis \eqref{sv_p4} that $\Ric\geq 0$, and Bishop-Gromov,
$$\frac{\VolB_{g(t)}(x_0,r)}{\om_n r^n}\leq 
\frac{\VolB_{g(t)}(x_0,\la r)}{\om_n (\la r)^n} \leq (1-\de),$$
as required.
\end{proof}

\section{Lifted flow lemma}
\label{lifted_flow_sect}

In order to prove Theorem \ref{main_thm},
we will be taking parabolic blow downs of the flow constructed in Theorem \ref{Thm:existence} and putting each into the following lemma to construct a sequence of incomplete Ricci flows, with uniform estimates, from which we can extract a limit flow.
An appropriate tangent flow of this will have incompatible properties, giving a contradiction.

The following lemma will trade a global flow $g(t)$ for a local flow $h(t)$ that is non-collapsed.



\begin{lem}
\label{lifted_flow_lemma}
Suppose $n\geq 3$ and $(M^n,g(t))$ is a complete  
Ricci flow for $t\in [0,\infty)$, and $x_0\in M$. Suppose that there exists 
$c_0>0$ so that 
\begin{enumerate}[(A)]
\item 
\label{part_a}
$\Ric_{g(t)}\geq 0$  for all $t\in [0,\infty)$; 
\item 
\label{part_b}
$|K|_{g(t)}\leq c_0 t^{-1}$ for all $t\in (0,\infty)$.
\end{enumerate}

Then there exist a constant $v_0>0$ depending only on $c_0$ and $n$, 
and a local diffeomorphism $\tau:\mathbb{B}_\pi\to M$, where $\mathbb{B}_\pi$ is 
the ball of radius $\pi$ in $\R^n$ centred at $0$,
such that 
$\tau(0)=x_0$ and 
the incomplete Ricci flow $h(t):=\tau^*(g(t))$ on $\mathbb{B}_\pi$ satisfies
\begin{enumerate}[(a)]
\item 
\label{part_a_conc}
$\Ric_{h(t)}\geq 0$ for all $t\geq 0$; 
\item 
\label{part_b_conc}
$|K|_{h(t)}\leq c_0 t^{-1}$ for all $t>0$; 
\item
\label{part_c_conc}
$B_{h(t)}(0,1)\Subset \mathbb{B}_\pi$ for all $t\in [0,c_0]$;
\item
\label{part_e_conc}
$\VolB_{h(t)}(0,r)\geq v_0 r^n$ for all $r\in (0,1]$ and $t\in [0,c_0]$;
\item
\label{comparable_conc}
for all $\de\in (0,c_0)$, there exists $c_1=c_1(\de,c_0, n)$ such that 
$$\frac{1}{c_1}\g\leq h(t) \leq c_1 \g \qquad \text{ on }\mathbb{B}_1,\text{ for all }t\in [\de,c_0],$$
where $\g$ is the Euclidean metric on $\mathbb{B}_\pi$.
\item
\label{deriv_conc}
for all $\de\in (0,c_0)$ and $k,l\in \N\union\{0\}$ there exists $c_{k,l}$ depending on $k$, $l$, $c_0$, $\de$  and $n$ such that
$$\left|\frac{\partial^l}{\partial t^l}\grad^k_\g h(t)\right|_\g\leq  c_{k,l} \qquad \text{ on }\mathbb{B}_1,\text{ for all }t\in [\de,c_0],$$
where $\grad^k_\g$ represents the $k$th covariant derivative with respect to the Levi-Civita connection of $\g$.
\end{enumerate}
\end{lem}

\begin{rmk}
\label{g_to_h_rmk}
Because $\tau$ is a local isometry from $(\mathbb{B}_\pi,h(t))$ to $(M,g(t))$, 
imposing additional hypotheses on the Ricci flow $g(t)$ in Lemma \ref{lifted_flow_lemma} often leads easily to the same conclusion for $h(t)$.
For example, 
if we strengthen hypothesis \eqref{part_a} to the condition
$\calr_{g(t)}\in \CPICo$ 
then conclusion \eqref{part_a_conc} strengthens to 
$\calr_{h(t)}\in \CPICo$. 
Similarly, if we strengthen hypothesis \eqref{part_a} further to the pinching condition
$\calr_{g(t)}-\e\, \Scal_{g(t)} \cdot \ci\in \CPICo$ 
for some $\ep\in (0,\frac{1}{n(n-1)})$ then 
$\calr_{h(t)}-\e\, \Scal_{h(t)} \cdot \ci\in \CPICo$.
\end{rmk}

Along the lines of Remark \ref{g_to_h_rmk}, we have the following lemma, 
which we prove at the end of this section.

\begin{lem}
\label{Ricci_balls_g_to_h_lem}
In the setting of Lemma \ref{lifted_flow_lemma},
if there exist $\si\in (0,1)$ and $\eta>0$ so that
$B_{g(0)}(x_0,\si)$ is an $\eta$-Ricci ball of radius $\si$  in the sense of Definition \ref{Ricci_ball_def}, then 
$B_{h(0)}(0,\si)$ is an $\eta$-Ricci ball  of radius $\si$.
\end{lem}

During the proof of Lemma \ref{lifted_flow_lemma}, 
we will require basic control on the exponential map that Hamilton has given in the following precise form. Note that the bound
$|K|_g\leq b_0$ implies a lower bound of 
$\pi/\sqrt{b_0}$ on the conjugate radius.
\begin{lem}[{\cite[Theorem 4.10 \& Corollary 4.11]{hamcptness}}]
\label{ham_exp_lem}
Suppose $(M^n,g)$ is a Riemannian manifold 
with $|K|_g\leq b_0$ for some constant $b_0>0$, 
and $B_g(x_0,\pi/\sqrt{b_0})\Subset M$ for some $x_0\in M$. Define a metric $h$ on $\mathbb{B}_{\pi/\sqrt{b_0}}\subset T_{x_0}M$ to be the pull back of $g$ under the exponential map at $x_0$.
Suppose further that $k\in \N$ and for each $j\in \{1,\ldots,k\}$,  $b_j$ is a constant for which
$$\left|\grad^j_{g} \calr_{g}\right|_{g}\leq  b_j.$$

Then there exist $\la\in (0,1)$ depending only on $n$, and constants $L_k$ depending only on $n$, $k$, and $b_0,\ldots b_k$, such that in the ball
$\mathbb{B}_{\la /\sqrt{b_0}}\subset T_{x_0}M$ we have
$$\frac{1}{2}\g\leq h \leq 2 \g,$$
where $\g := g|_{x_0}$, 
and 
$$|\grad^k_\g h|_\g \leq L_k.$$
\end{lem}



\begin{proof}[{Proof of Lemma \ref{lifted_flow_lemma}}.]
Let $\la\in (0,1)$ be the dimensional constant from Lemma \ref{ham_exp_lem}, and define $t_0:=\frac{c_0}{\la^2}>c_0$. Although \eqref{part_c_conc}, \eqref{part_e_conc} and \eqref{comparable_conc}, \eqref{deriv_conc} are stated on the intervals $[0,c_0]$ and $[\delta,c_0]$ resp., for technical reasons we will establish them on the longer time intervals $[0,t_0]$ and $[\delta,t_0]$ resp.

We begin by recalling that the Ricci flow has a smoothing effect provided the curvature remains bounded. As a precise instance of this, the boundedness of the curvature of $g(t)$ as $t$ ranges over compact subintervals of $(0,t_0]$, coming from hypothesis \eqref{part_b}, implies estimates of the form
\beql{deriv_ests}
\left|\frac{\partial^l}{\partial t^l}\grad^k_{g(t)} \calr_{g(t)}\right|_{g(t)}\leq  C_{k,l} \qquad \text{ on }M \text{ for }t\in [\de,t_0],
\eeq
where $\de\in (0, t_0)$,
$k,l\in\N\union\{0\}$,
$\grad^k_{g(t)}$ represents the $k$th covariant derivative with respect to the Levi-Civita connection of $g(t)$, and $C_{k,l}$ depends only on $k$, $l$, $\de$, $c_0$ and $n$.
This  can be derived from  local arguments \cite{shi1989}.

After making an isometric identification between $(T_{x_0}M,g(t_0))$ and $\R^n$, we define the map $\tau:\mathbb{B}_\pi\to M$ to be the exponential map at $x_0$, with respect to 
$g(t_0)$, where $\mathbb{B}_\pi$ is the ball of radius $\pi$ in $\R^n$.  

By hypothesis \eqref{part_b}, and the definition of $t_0$, we know that 
$(M, g(t_0))$ has every sectional curvature bounded above by $b_0:=\la^2\leq 1$. This implies that the conjugate radius at $x_0$ is at least $\pi$. Consequently, the map $\tau$ is a local diffeomorphism, and so 
we can define a Ricci flow $(\mathbb{B}_\pi,h(t))$ for $t\geq 0$ by 
$h(t):=\tau^*(g(t))$. This flow inherits the curvature decay 
\eqref{part_b} to give \eqref{part_b_conc}, and inherits the non-negativity of the Ricci curvature \eqref{part_a} to give \eqref{part_a_conc}. 
It inherits the estimates on the derivatives of curvature in \eqref{deriv_ests}, i.e.
\beql{deriv_ests2}
\left|\frac{\partial^l}{\partial t^l}\grad^k_{h(t)} \calr_{h(t)}\right|_{h(t)}\leq  C_{k,l} \qquad \text{ on }\mathbb{B}_\pi \text{ for }t\in [\de, t_0].
\eeq
Restricting to  $l=0$ and $t=t_0$, we may apply Lemma \ref{ham_exp_lem} to give
\beq
\label{exp_map_basic}
 \frac{1}{2}\g\leq h(t_0) \leq 2 \g \qquad \text{ on } \mathbb{B}_1 
\eeq
and
\beql{h_derivs}
 |\grad^k_\g (h( t_0))|_\g \leq C(k,n,c_0) \quad\text{ throughout }\mathbb{B}_1. 
\eeq
By inspection of the Ricci flow equation, the Ricci non-negativity implies that 
$\pl{h}{t}=-2\Ric_h\leq 0$, so the 
lengths of $C^1$ curves in $\mathbb{B}_\pi$  are decreasing in time. 
In particular, we have 
\beql{time1_balls}
B_{h(t)}(0,1)\subset B_{h(t_0)}(0,1)
\Subset B_{h(t_0)}(0,\pi)=\mathbb{B}_\pi
\eeq
for all $t\in [0, t_0]$. 
Indeed, given $x\in B_{h(t)}(0,1)$, we can pick a $C^1$ curve $\ga:[0,1]\to \mathbb{B}_\pi$ with $\ga(0)=0$, $\ga(1)=x$ and $L_{h(t)}(\ga)<1$, and use the length decreasing property of the flow to give
$$L_{h( t_0)}(\ga)\leq L_{h(t)}(\ga)<1,$$
which implies that $x\in B_{h( t_0)}(0,1)$ as required.
Inclusion \eqref{time1_balls} establishes \eqref{part_c_conc} on the longer time interval $[0,t_0]$. It implies that points in $B_{h(t)}(0,1)$ can be reached from the origin by minimising geodesics.


A consequence of the Ricci flow equation for $h(t)$ and the curvature bounds 
\eqref{part_a_conc} and \eqref{part_b_conc} is that for any $\de\in (0,t_0)$, there exists $C>0$ depending only on $\de$, $n$ and $c_0$ such that 
$$-C h(t) \leq \pl{h}{t} \leq 0 \quad\text{ on }\mathbb{B}_\pi,\text{ for all }t\in [\de, t_0].$$
Therefore, for a new $C$ with the same dependencies, we have
\beql{h_comparable}
h( t_0)\leq h(t) \leq C h(t_0)\quad\text{ on }\mathbb{B}_\pi,\text{ for all }t\in [\de, t_0].
\eeq

Coupled with \eqref{exp_map_basic}, this gives 
that there exists $c_1=c_1(\de,c_0, n)$ such that 
\beql{comparable_to_t0}
\frac{1}{c_1}\g\leq h(t) \leq c_1 \g \qquad \text{ on }\mathbb{B}_1,\text{ for all }t\in [\de,t_0],
\eeq
i.e. the 
conclusion \eqref{comparable_conc}
on the longer time interval $[\de,t_0]$.

The combination of the comparability of $h(t)$ and $\g$, the estimates on the derivatives of curvature from \eqref{deriv_ests2} 
and the control \eqref{h_derivs} on the derivatives of $h$
imply conclusion \eqref{deriv_conc}, 
for $l=0$, even on the longer time interval $[\de,t_0]$,
as explained by Hamilton \cite[Lemma 2.4]{hamcptness}. 
The case of general $l$ follows by differentiating the equation, as explained 
by Hamilton in the proof of \cite[Lemma 2.4]{hamcptness}.

We now have to consider ball inclusions going 
forwards in time rather than backwards, using the eternally shrinking balls lemma \ref{eternally_shrinking_balls_lem} with $R=1$, $T=t_0$, $M=\mathbb{B}_\pi$, $x_0=0\in \mathbb{B}_\pi$, and with $g(t)$ there equal to $h(t)$ here. 
We find that there exists $r_0\in (0,1)$ depending only on $n$ and $c_0$ such that 
\begin{equation}
\label{back_ball_inclusion}
B_{h({t_0})}(0,r_0)\subset B_{h(0)}(0,1).
\end{equation}
More generally, for fixed $s\in [0,{t_0}]$ we could apply this lemma to the shifted flow $h(t+s)$ for $0\leq t\leq t_0-s\leq t_0$ to obtain
\begin{equation}
\label{back_ball_inclusion2}
B_{h({t_0})}(0,r_0)\subset B_{h(s)}(0,1).
\end{equation}

By 
\eqref{comparable_to_t0} at $t=t_0$,
the volume of $B_{h(t_0)}(0,r_0)$ will enjoy a uniform positive lower bound
$$\VolB_{h(t_0)}(0,r_0) = \Vol_{h(t_0)}(\mathbb{B}_{r_0})\geq v_0$$
for some $v_0>0$ depending only on $c_0$ and $n$.
However, in addition, by \eqref{back_ball_inclusion2}
we have
$$\Vol_{h(t_0)}\left(B_{h(s)}(0,1)\right) 
\geq \VolB_{h(t_0)}(0,r_0)\geq v_0$$
for all $s \in [0,t_0].$
The non-negativity of the scalar curvature forces the volume of any fixed set to decrease in time under the Ricci flow (see e.g.~\cite[(2.5.7)]{RFnotes}), so this implies
$$\VolB_{h(s)}(0,1)\geq v_0,
\quad\text{ for all }s\in [0,t_0].$$
We can now appeal to Bishop-Gromov, using the fact that $\Ric\geq 0$, to give
\eqref{part_e_conc} on the longer time interval $[0,t_0]$. 
\end{proof}

\begin{proof}[{Proof of Lemma \ref{Ricci_balls_g_to_h_lem}}]

Observe that $\tau$ maps $B_{h(0)}(0,\si)\subset \mathbb{B}_\pi$ to $B_{g(0)}(x_0,\si)\subset M$, which is an $\eta$-Ricci ball, so $B_{h(0)}(0,\si)$ is 
also an $\eta$-Ricci ball,
as claimed in the lemma.
Indeed, a minimising geodesic between $0$ and $x$ in $(\mathbb{B}_\pi,h(0))$
is mapped by the local isometry $\tau$ to a curve of the same length in $(M,g(0))$.
Although we do not need it, we remark that the image of 
$B_{h(0)}(0,\si)$ under $\tau$ will be precisely $B_{g(0)}(x_0,\si)$ since 
any length minimising geodesic between $x_0$ and $  p \in  B_{g(0)}(x_0,\sigma) $ lifts to a curve  of the same length from $0$ to a point in $B_{h(0)}(0,\sigma) $.
\end{proof}

\section{Harnack conclusions with mild curvature hypotheses}
\label{harnack_sub_sect}

In this section we prove the following  extension of Theorem  \ref{lifted_flow_app_intro}.

\begin{thm}
\label{lifted_flow_app}
Suppose $n\geq 3$ and $(M^n,g(t))$ is a complete  Ricci flow for $t\in [0,\infty)$, and  $x_0\in M$. Suppose that there exists $c_0>0$  so that 
\begin{enumerate}[(A)]
\item 
\label{lifted_flow_app_part_a}
$\Ric_{g(t)}\geq 0$  for all $t\in [0,\infty)$; 
\item 
\label{lifted_flow_app_part_b}
$|K|_{g(t)}\leq c_0 t^{-1}$ for all $t\in (0,\infty)$.
\end{enumerate}
Then the following two conclusions hold independently:
\begin{enumerate}[(i)]
\item 
\label{part_i_LFL}
If $\Ric_{g(0)}(x_0)>0$, then 
\beql{Ric_pos_conc}
\liminf_{t\to\infty} \ t\,\Scal_{g(t)}(x_0)>0
\eeq
\item 
\label{part_ii_LFL}
If there exists $\ep\in (0,\frac{1}{n(n-1)})$ such that 
\beql{another_pinching_hyp}
\calr_{g(t)}-\e\, \Scal_{g(t)} \cdot \ci\in \CPICo,
\eeq
for all $t\geq 0$ and throughout $M$, 
then 
\beql{tScal_null}
\liminf_{t\to\infty} \ t\,\Scal_{g(t)}(x_0)=0.
\eeq
\end{enumerate}
\end{thm}
Part \eqref{part_i_LFL} of this theorem is known for $n=2$ by Hamilton's Harnack inequality, as alluded to in the introduction. Part \eqref{part_ii_LFL} makes sense only for $n\geq 3$.

\begin{proof}
%
%
%
Define $\la:=\liminf_{t\to\infty} \ t\,\Scal_{g(t)}(x_0)$. By hypotheses 
\eqref{lifted_flow_app_part_a} and \eqref{lifted_flow_app_part_b}, we know that $\la\in [0,\infty)$.
Pick $t_m\to\infty$ so that $t_m\,\Scal_{g(t_m)}(x_0)\to \la$.

We might be tempted to apply the lifted flow lemma \ref{lifted_flow_lemma} to the Ricci flow $g(t)$. Instead, for each $m\in\N$, we define a blown-down Ricci flow 
$$g_m(t)=\frac{1}{R_m^2} g(R_m^2t),$$ 
where $R_m:=\left(\frac{t_m}{c_0}\right)^\half$,
which still satisfies \eqref{lifted_flow_app_part_a} and \eqref{lifted_flow_app_part_b} above, but for which now 
\beqa
\label{scal_to_zero}
\Scal_{g_m(c_0)}(x_0) &= \Scal_{\frac{1}{R_m^2} g(R_m^2 c_0)}(x_0)
=R_m^2 \Scal_{g(R_m^2 c_0)}(x_0)=\frac{1}{c_0}t_m \Scal_{g(t_m)}(x_0)\\
& \to \frac{\la}{c_0},
\eeqa
as $m\to\infty$,
and apply Lemma \ref{lifted_flow_lemma} to $g_m(t)$ instead. 
By deleting finitely many terms of these sequences, we may assume that $R_m>1$ for all $m\in\N$, i.e. that $g_m(t)$ is genuinely a blow-down rather than a blow-up.


The output of the lemma is a constant $v_0>0$ depending only on $c_0$ and $n$, 
and a sequence of local diffeomorphisms $\tau_m:\mathbb{B}_\pi\to M$, 
such that $\tau_m(0)=x_0$, and 
a sequence of local Ricci flows $(\mathbb{B}_\pi,h_m(t))$, 
where $h_m(t):=\tau_m^*(g_m(t))$, on $\mathbb{B}_\pi$ 
satisfying
\begin{enumerate}[(a)]
\item 
\label{part_a_conc2_new}
$\Ric_{h_m(t)}\geq 0$ for all $t\geq 0$; 
\item 
\label{part_b_conc2_new}
$|K|_{h_m(t)}\leq c_0 t^{-1}$ for all $t>0$;
\item
$B_{h_m(t)}(0,1)\Subset \mathbb{B}_\pi$ for all $t\in [0,c_0]$;
\item
\label{part_e_conc2_new}
$\VolB_{h_m(t)}(0,r)\geq v_0 r^n$ for all $r\in (0,1]$ and $t\in [0,c_0]$.
\item
\label{comparable_conc2_new}
for all $\de\in (0,c_0)$, there exists $c_1=c_1(\de,c_0, n)$ such that 
$$\frac{1}{c_1}\g\leq h_m(t) \leq c_1 \g \qquad \text{ on }\mathbb{B}_1,\text{ for }t\in [\de,c_0];$$
\item
\label{deriv_conc2_new}
for all $\de\in (0,c_0)$ and $k,l\in \N\union\{0\}$ there exists $c_{k,l}$ depending on $k$, $l$, $c_0$, $\de$ and $n$ such that
$$\left|\frac{\partial^l}{\partial t^l}\grad^k_\g h_m(t)\right|_\g\leq  c_{k,l} \qquad \text{ on }\mathbb{B}_1,\text{ for }t\in [\de,c_0];$$
\item 
$\Scal_{h_m(c_0)}(0)\to \frac{\la}{c_0}$ as $m\to\infty$,
\end{enumerate}
where the final part is the translation of \eqref{scal_to_zero}.

The estimates of parts \eqref{comparable_conc2_new} and \eqref{deriv_conc2_new} above
give us enough control to apply Ascoli-Arzel\`a directly to the metric tensors $h_m(t)$:
We can pass to a subsequence in $m$ and obtain a limit Ricci flow 
$(\mathbb{B}_1,h_\infty(t))$,
$t\in (0,c_0]$, satisfying 

\begin{enumerate}[(i)]
\item 
\label{part_a_conc3_new}
$\Ric_{h_\infty(t)}\geq 0$
for all $t\in (0,c_0]$; 
\item 
\label{part_b_conc3_new}
$|K|_{h_\infty(t)}\leq c_0 t^{-1}$ for all $t\in (0,c_0]$;
\item
$B_{h_\infty(t)}(0,s)\Subset \mathbb{B}_1$ for all $t\in (0,c_0]$ and $s\in (0,1)$;
\item
\label{part_e_conc3_new}
$\VolB_{h_\infty(t)}(0,r)\geq v_0 r^n$ for all $r\in (0,1)$ and $t\in (0,c_0]$;
\item 
\label{scal_hinf}
$\Scal_{h_\infty(c_0)}(0)= \frac{\la}{c_0}$.
\end{enumerate}

We  now blow up the flow $(\mathbb{B}_1,h_\infty(t))$ parabolically. 
For each $m\in\N$, we define new flows $(\mathbb{B}_1,H_m(t))$, for $t\in (0,m^2 c_0]$,
by 
$$H_m(t)=m^2 h_\infty(m^{-2}t).$$
%
The properties \eqref{part_a_conc3_new} to \eqref{part_e_conc3_new} translate to
\begin{enumerate}[(I)]
\item 
\label{part_a_conc4_new}
$\Ric_{H_m(t)}\geq 0$
for all $t\in (0,m^2c_0]$; 
\item 
\label{part_b_conc4_new}
$|K|_{H_m(t)}\leq c_0 t^{-1}$ for all $t\in (0,m^2c_0]$;
\item
\label{part_c_conc4_new}
$B_{H_m(t)}(0,s)\Subset \mathbb{B}_1$ for all $t\in (0,m^2 c_0]$ and $s\in (0,m)$;
\item
\label{part_e_conc4_new}
$\VolB_{H_m(t)}(0,r)\geq v_0 r^n$ for all $r\in (0,m)$ and $t\in (0,m^2c_0]$.
\end{enumerate}

We can now appeal to Cheeger-Gromov-Hamilton compactness \cite{hamcptness}, passing to a subsequence in $m$ to give convergence
$$(\mathbb{B}_1,H_m(t),0)\to (\m,H_\infty(t),p)\quad\text{ as }m\to\infty,$$
for some smooth manifold $\m^n$, Ricci flow $H_\infty(t)$ on $\m$, $t>0$, and $p\in\m$.
Here we use the curvature control \eqref{part_b_conc4_new} and the consequence of 
\eqref{part_e_conc4_new} that 
$\VolB_{H_m(1)}(0,1)\geq v_0$ for sufficiently large $m$, while 
\eqref{part_c_conc4_new} is required to be sure of a well-defined \emph{complete} 
limit. 

The properties \eqref{part_a_conc4_new}, \eqref{part_b_conc4_new} and \eqref{part_e_conc4_new} translate to
\begin{enumerate}[(I)]
\item 
\label{part_a_conc4_Hinf}
$\Ric_{H_\infty(t)}\geq 0$ throughout $\m$ and for all $t>0$;
\item
\label{part_b_conc4_Hinf}
$|K|_{H_\infty(t)}\leq c_0 t^{-1}$ for all $t>0$;
\item 
\label{part_e_conc4_Hinf}
$\VolB_{H_\infty(t)}(0,r)\geq v_0 r^n$ for all $r>0$ and $t>0$, so $\mathrm{AVR}(H_\infty(t))\geq \frac{v_0}{\om_n}>0$ for all $t>0$.
\end{enumerate}

We have managed to extract a complete Ricci flow $(\m,H_\infty(t))$ with good curvature control and positive asymptotic volume ratio from our original flow $g(t)$. What we do with it will depend on whether we are proving Part \eqref{part_i_LFL} or Part \eqref{part_ii_LFL} of the theorem.

{\bf Proof of Part \eqref{part_i_LFL}.}

Let's assume that both $\Ric_{g(0)}(x_0)>0$ and 
$\la:=\liminf_{t\to\infty} \ t\,\Scal_{g(t)}(x_0)=0$, and aim for a contradiction.
Because $\la=0$, Property \eqref{scal_hinf} above tells us that 
$\Scal_{h_\infty(c_0)}(x_0)= 0$. As the scalar curvature satisfies the equation
\beql{scal_evol_eq}
\left(\pl{}{t}-\lap_{h_\infty(t)}\right)\Scal_{h_\infty(t)}=2|\Ric_{h_\infty(t)}|^2,
\eeq
see e.g. \cite[Proposition 2.5.4]{RFnotes},
the strong maximum principle tells us that $\Scal_{h_\infty(c_0)}\equiv 0$ throughout 
$\mathbb{B}_1$ and for all $t\in (0,c_0]$.
By \eqref{scal_evol_eq} this then forces the entire flow $h_\infty(t)$ to be Ricci flat, and in particular static. This property then extends first to $H_m(t)$ and then $H_\infty(t)$. That is, $H_\infty(t)=H:=H_\infty(1)$ for all $t>0$ where $H$ is a complete Ricci flat metric on $\m$. 

We now turn to the consequences of having $\Ric_{g(0)}(x_0)>0$.
By smoothness of $g(0)$, this implies 
that there exist $r>0$ and $\eta>0$ such that 
$B_{g(0)}(x_0,r)$ is an $\eta$-Ricci ball of radius $r$.

By going back and making a once and for all parabolic scaling of $g(t)$, we may assume that $r=1$.

By the scale-invariant property of $\eta$-Ricci balls, we can transfer this property to the rescaled flows $g_m(t)$. We learn that 
$B_{g_m(0)}(x_0,\frac{1}{R_m})$ is an $\eta$-Ricci ball of radius $\frac{1}{R_m}$.
By Lemma \ref{Ricci_ball_def}
we find that $B_{h_m(0)}(0,\frac{1}{R_m})$ is also an $\eta$-Ricci ball.
Lemma \ref{Miles_Ricci_balls_lem} then gives us $\de\in (0,\half)$ depending only on 
$n$ and $\eta$ such that $\VolB_{h_m(0)}(0,\frac{1}{R_m})\leq (1-2\de)\om_n (\frac{1}{R_m})^n$. 

By Bishop-Gromov and the fact that $\Ric_{h_m(0)}\geq 0$, this then implies that 
$$\VolB_{h_m(0)}(0,r)\leq (1-2\de)\om_n r^n$$
for all $r\in [\frac{1}{R_m},1)$.

We can then apply Lemma \ref{Ricci_balls_vol_suppression_lem} to each flow $h_m(t)$ to find that there exist $\de=\de(n,\eta)\in (0,1)$, 
and $\rho=\rho(n,\eta,c_0)>0$ such that 
$$\VolB_{h_m(t)}(0,r)\leq (1-\de)\om_n r^n$$
for all $r\in [\frac{1}{R_m},1)$ and $t\in [0,\rho r^2]$.
Note that Lemma \ref{Ricci_balls_vol_suppression_lem} would 
initially apply for  $t\in [0, \min(\rho r^2,c_0)]$, but by insisting that  
$\rho\leq c_0$ we have that $\rho r^2\leq \rho\leq c_0$.

These estimates pass to the limit $m\to\infty$ to give
\beql{h_inf_vol_ests_new}
\VolB_{h_\infty(t)}(0,r)\leq (1-\de)\om_n r^n
\eeq
for all $r\in (0,1)$ and $t\in (0, \rho r^2]$.
Note that it is crucial that we have achieved this control for $r$ as small as we like.
After parabolically rescaling up to the flows $H_m(t)$, we obtain
\beql{Hm_vol_ests_new}
\VolB_{H_m(t)}(0,r)\leq (1-\de)\om_n r^n
\eeq
for all $r\in (0,m)$ and $t\in (0, \rho r^2]$.
This passes to the limit $m\to\infty$ to give
$$\VolB_{H_\infty(t)}(0,r)\leq (1-\de)\om_n r^n$$
for all $r>0$ and $t\in (0, \rho r^2]$.

Recall now that the assumption $\la=0$ implied that $H_\infty(t)=H$ for all $t>0$.
This then gives us
$$\VolB_{H}(0,r)\leq (1-\de)\om_n r^n$$
for all $r>0$, which is impossible on a smooth Riemannian manifold as $r\downto 0$, 
giving a contradiction.

{\bf Proof of Part \eqref{part_ii_LFL}.}

In Part \eqref{part_ii_LFL}, the pinching condition  \eqref{another_pinching_hyp} holds.
This pinching condition transfers immediately to the rescalings $g_m(t)$, and keeping in mind Remark \ref{g_to_h_rmk}, it also transfers to $h_m(t)$.
From there, it passes to $h_\infty(t)$ in the limit $m\to\infty$, then transfers to the blow-ups $H_m(t)$, and finally to $H_\infty(t)$ in the limit $m\to\infty$.
Thus we have 
$$\calr_{H_\infty(t)}-\e\, \Scal_{H_\infty(t)} \cdot \ci\in \CPICo
\quad\text{ throughout }\m\text{ and for all }t>0.$$
Keeping in mind that we constructed $H_\infty(t)$ to be complete and of positive asymptotic volume ratio, Theorem \ref{DSS_AVR_thm} tells us that this complete flow must be static Euclidean space.

Suppose for a contradiction that \eqref{tScal_null} fails, i.e. that $\la>0$.
Then there exists $T_0>0$ such that $t\, \Scal_{g(t)}(x_0)\geq \frac{\la}{2}$ for all 
$t\geq T_0$.
When we blow down to the Ricci flows $g_m(t)$, this control will now hold for 
$t\geq \frac{T_0}{R_m^2}$. Keeping in mind Remark \ref{g_to_h_rmk}, it also transfers to $h_m(t)$. When we take the limit $m\to\infty$, we obtain 
$$t\, \Scal_{h_\infty(t)}(0)\geq \frac{\la}{2}\qquad \text{ for all }t\in (0,c_0].$$
We get the same control for the blow-ups $H_m(t)$, this time for all $t\in (0,m^2 c_0]$.
In the limit $m\to\infty$ we obtain
$$t\, \Scal_{H_\infty(t)}(p)\geq \frac{\la}{2}>0\qquad \text{ for all }t>0.$$
Since we already showed that $(\m,H_\infty(t))$ is static Euclidean space, this is a contradiction.
\end{proof}

\section{Proof of the main theorem \ref{main_thm}}
\label{main_thm_pf}

Almost all of the work required to prove the main theorem has been done in Theorem 
\ref{lifted_flow_app}.

\begin{proof}
Suppose for a contradiction that $(M,g)$ is a complete PIC1 pinched manifold, as in the theorem, that is neither flat nor compact.

Apply Theorem \ref{Thm:existence} to obtain a complete Ricci flow $g(t)$ for $t\geq 0$,
with $g(0)=g$, and $\e_0\in (0,\frac{1}{n(n-1)})$, with 
\begin{compactenum}[(a)]
\item 
\label{conc_a}
$\calr_{g(t)}-\e_0\, \Scal_{g(t)} \cdot \ci\in \CPICo$; 
\item 
\label{conc_b}
$|K|_{g(t)}\leq c_0 t^{-1}$.
\end{compactenum}
By Remark \ref{pinch_trace}, keeping in mind that $\ep_0<\frac{1}{n(n-1)}$, 
we find that $\Scal_{g(t)}\geq 0$, $\calr_{g(t)}\in \CPICo$, and 
$\Ric_{g(t)}\geq 0$.

By the non-flatness of $(M,g)$, there exists a point $x_0\in M$ at which 
$|\calr|_g (x_0)>0$, so by \cite[Lemma A.2]{LeeTopping2022_PIC1} we have $\Scal_g (x_0)>0$.
Appealing to the pinching hypothesis \eqref{main_thm_pinching_hyp}, this implies
$\Ric_g (x_0)>0$.

We are now in a position to apply Theorem \ref{lifted_flow_app}.
Part \eqref{part_i_LFL} of that theorem tells us that 
$$\liminf_{t\to\infty} \ t\,\Scal_{g(t)}(x_0)>0,$$ 
whereas Part \eqref{part_ii_LFL} of that theorem gives the contradictory statement
$$\liminf_{t\to\infty} \ t\,\Scal_{g(t)}(x_0)=0.$$ 
\end{proof}

\section{An improved gap theorem}

\begin{thm}\label{thm:new-gap-dyn}
Suppose $(M^n,g_0)$, $n\geq 3$, is a complete manifold  that admits an immortal complete  Ricci flow $(M^n,g(t))$ for $t\in [0,\infty)$ starting from $g_0$ such that
\begin{enumerate}[(i)]
    \item \label{gap_i} $\mathcal{R}_{g(t)}  \in  \CPICo$ for all $t\in [0,+\infty)$;
    \item \label{gap_ii} $|K|_{g(t)}\leq c_0 t^{-1}$ for some $c_0>1$ and for all $t\in (0,+\infty)$;
    \item \label{gap_iii} $\liminf_{t\to+\infty} t\, \Scal_{g(t)}(x_1)=0$ for some $x_1\in M$,
\end{enumerate}
then $g_0$ is flat.
\end{thm}

\begin{proof}
Let's suppose that $g_0$ is \emph{not} flat, and aim for a contradiction.
By assumption \eqref{gap_i} we have $\Scal_{g(t)}\geq 0$ throughout.
We may assume $\Scal_{g(t)}>0$ for all $t>0$ since otherwise
the strong maximum principle implies $\Scal_{g(t)}\equiv 0$, 
which then implies the  flatness of $g(t)$, including $g_0$, by assumption \eqref{gap_i} (see \cite[Lemma A.2]{LeeTopping2022_PIC1}) giving a contradiction. 
By lifting $M$ to its universal cover, we may assume that $M$ is simply connected. 
Fix $s_0>0$ and consider the de-Rham decomposition \cite{deRham1952}
\begin{equation}
    (M,g(s_0))=\prod_{i=1}^N (M_i,g_i)
\end{equation}
of $(M,g(s_0))$ into irreducible components.
By the existence and uniqueness of Ricci flow in the bounded curvature case 
\cite{shi1989, ChenZhu}, $(M,g(t))$ splits isometrically as $\prod_{i=1}^N (M_i,g_i(t))$ for all $t\geq s_0$. 
For each $i$ with $\dim(M_i)\geq 3$, we have $\calr_{g_i(t)}\in\CPICo$, while if 
$\dim(M_i)=2$ then $g_i(t)$ has non-negative Gauss curvature.
Assumptions \eqref{gap_ii} and \eqref{gap_iii} carry over to each Ricci flow 
$t\mapsto g_i(s_0+t)$.
Since $\Scal_{g(s_0)}>0$, there exists $i_0\in \{1,\ldots,N\}$ such that $\Scal_{g_{i_0}(s_0)}>0$ somewhere on $M_{i_0}$. By the strong maximum principle, we have $\Scal_{g_{i_0}(t)}>0$ on $M_{i_0}\times (s_0,\infty)$. 
We must have $\mathrm{dim}(M_{i_0})\geq 3$ because if $\dim(M_{i_0})=2$ then 
Hamilton's Harnack inequality \cite{Ham_RF_Harnack} will contradict assumption \eqref{gap_iii}, as discussed in Section \ref{pinch_thm_sect}. 
Also, $M_{i_0}$ must be non-compact, otherwise the positivity of the scalar curvature of 
$g_{i_0}(t)$ will force blow-up in finite time.
Thus we can replace $(M,g(t))$ with $(M_{i_0},g_{i_0}(s_0+t))$, and the new $(M,g(t))$
satisfies all the hypotheses of Theorem \ref{thm:new-gap-dyn}, 
but additionally has positive scalar curvature curvature for $t>0$,
is simply connected, and has initial 
data $(M,g(0))$ that is irreducible, which will help us derive a contradiction.

Fix $t_0>0$ sufficiently small so that $(M,g(t_0))$ remains irreducible. 
We may further assume that $\left(M,g(t_0)\right)$ is non-symmetric, since otherwise $\Scal_{g(t_0)}\equiv c_1$ on $M$ for some constant $c_1> 0$,
which would force blow up in finite time as above.
Thus, we must have $\left(M,g(t_0)\right)$ is simply connected, non-symmetric and irreducible. By Berger’s holonomy classification theorem \cite{Berger}, $\mathrm{Hol}(M,g(t_0))$ is either $\mathrm{SO}(n)$ or possibly if $n$ is even, $\mathrm{U}(n/2)$. This is because the other options would be Ricci flat  
(and hence flat by assumption \eqref{gap_i})
or Einstein (and hence with $\Scal_{g(t_0)}\equiv c_1>0$ as above).

We claim that $\Ric_{g(t_0)}(x)>0$ for all $x\in M$. Otherwise, there exist $x_0\in M$
and $v\in T_{x_0}M$ with $|v|=1$ such that $\Ric(v,v)=0$ at $(x_0,t_0)$. Suppose $\mathrm{Hol}(M,g(t_0))=\mathrm{SO}(n)$.
We let $\{e_i\}_{i=1}^n$ be an orthonormal frame at $(x_0,t_0)$ such that $e_1=v$. 
By the (weakly) 
PIC1 assumption \eqref{gap_i},  we deduce from $\Ric(e_1,e_1)=\sum_{i=2}^n \calr(e_1,e_i,e_1,e_i)=0$ that 
\beql{two_sectionals}
\calr(e_1,e_i,e_1,e_i)+\calr(e_1,e_j,e_1,e_j)=0
\eeq
for all $j,i>1$ with $i\neq j$. 
This is because we can write the Ricci curvature $\Ric(e_1,e_1)$ as the sum of 
terms of the form in \eqref{two_sectionals},
each of which is non-negative by virtue of the (weakly) PIC1 condition;
see \cite[Section 2]{PIC1_survey}. 
Mimicking the proof of \cite[Proposition 9]{BS-2}, by applying \cite[Proposition 8]{BS-2} to $(M,g(t))\times \mathbb{S}^1$, we find that \eqref{two_sectionals} is invariant under the action of $\mathrm{SO}(n)$ on $\{e_1,\ldots e_n\}$. Therefore for all $i,j,k$ distinct
we have 
$$\calr(e_k,e_i,e_k,e_i)+\calr(e_k,e_j,e_k,e_j)=0.$$
By summing over $i,j$ with $i,j,k$ pairwise distinct, this implies $\Ric(e_k,e_k)=0$ for all $1\leq k\leq n$. This contradicts $\mathrm{Scal}_{g(t_0)}(x_0)>0$.

It remains to consider the case $\mathrm{Hol}(M,g(t_0))=\mathrm{U}(n/2)$, where now $g(t_0)$ is K\"ahler.
The argument is similar. We fix an orthonormal frame $\{e_i,Je_i\}_{i=1}^{n/2}$ such that $e_1:=v$. 
Thanks to  the (weakly) PIC1 condition \eqref{gap_i} and $\Ric(e_1,e_1)=0$, we have 
\begin{equation}
\left\{
\begin{array}{ll}
    \calr(e_1,Je_1,e_1,Je_1)+\calr(e_1,e_k,e_1,e_k)= 0;\\[1mm]
  \calr(e_1,Je_1,e_1,Je_1)+\calr(e_1,Je_k,e_1,Je_k)=0;\\[1mm]
  \calr(e_1,e_k,e_1,e_k)+\calr(e_1,Je_k,e_1,Je_k)=0
\end{array}
    \right.
\end{equation}
for all $k>1$. Consider the transformation of $T_{x_0}M$ given by $(e_1,Je_1,e_k,Je_k)\mapsto (e_k,Je_k,e_i,Je_i)$ for $k>1$ and $i\neq k$, which is an element of $\mathrm{U}(n/2)$. We apply \cite[Proposition 8]{BS-2} to $(M,g(t))\times \mathbb{S}^1$ again as in the case of $\mathrm{Hol}(M,g(t_0))=\mathrm{SO}(n)$, with now $\mathrm{Hol}(M,g(t_0))=\mathrm{U}(n/2)$, to conclude that 
\begin{equation}
\left\{
\begin{array}{ll}
    \calr(e_k,Je_k,e_k,Je_k)+\calr(e_k,e_i,e_k,e_i)= 0;\\[1mm]
  \calr(e_k,Je_k,e_k,Je_k)+\calr(e_k,Je_i,e_k,Je_i)=0;\\[1mm]
  \calr(e_k,e_i,e_k,e_i)+\calr(e_k,Je_i,e_k,Je_i)=0
\end{array}
    \right.
\end{equation}
for $i\neq k$. Hence, 
$$\calr(e_k,Je_k,e_k,Je_k)= \calr(e_k,Je_i,e_k,Je_i)=\calr(e_k,e_i,e_k,e_i) =0$$ 
for all $1\leq k\leq n/2$ and $i\neq k$. Thus,
\begin{equation}
    \Ric(e_k,e_k):= \calr(e_k,Je_k,e_k,Je_k)+\sum_{i\neq k}\calr(e_k,Je_i,e_k,Je_i)+\calr(e_k,e_i,e_k,e_i)=0. 
\end{equation}
Since $\Ric$ is $J$-invariant from K\"ahlerity, we conclude that 
$\Scal_{g(t_0)}(x_0)=0$ which is impossible. 
To summarize, we have shown that $\Ric_{g(t_0)}(x)>0$ for all $x\in M$. But this  implies  conclusion \eqref{part_i_LFL}
of Theorem~\ref{lifted_flow_app} must hold, which contradicts  assumption \eqref{gap_iii} of the theorem we are proving. 
%
\end{proof}

In \cite{ChanLee2025}, a gap phenomenon for Riemannian manifolds with non-negative complex sectional curvature was obtained by studying the long-time behaviour of the Ricci flow, based on Brendle's Harnack inequality \cite{Brendle2009}. As an application of  Theorem \ref{thm:new-gap-dyn}, we improve the gap Theorem in \cite{ChanLee2025} to the weaker curvature condition PIC1.
\begin{proof}[Proof of Theorem~\ref{thm:gap}]
For $\mathcal{R}_{g_0}  \in  \CPICt$ instead of $\CPICo$ the result was shown in \cite[Theorem 1.1]{ChanLee2025}. To extend the result to $\mathcal{R}_{g_0}  \in  \CPICo$ we follow the strategy in \cite[Theorem 1.1]{ChanLee2025}, but replacing the implications of Brendle's Harnack inequality by Theorem~\ref{thm:new-gap-dyn}. 

By \cite[Proposition 5.1]{ChanLee2025}, provided $\e_n$ is sufficiently small, there exists a solution to Ricci flow $g(t),t\in [0,+\infty)$ starting from $g_0$ and $c_0(n)>0$ such that 
\begin{enumerate}[(i)]
    \item \label{1_10_i} $\mathcal{R}_{g(t)}  \in  \CPICo$;
    \item \label{1_10_ii} $|K|_{g(t)}\leq c_0 t^{-1}$,
\end{enumerate}
for all $(x,t)\in M\times (0,+\infty)$. We claim that 
\begin{equation}\label{eqn:claim}
    \liminf_{t\to+\infty} t\, \Scal_{g(t)}(x_0)=0
\end{equation}
for some $x_0\in M$. Assuming \eqref{eqn:claim}, Theorem~\ref{thm:new-gap-dyn} implies then that $(M,g_0)$ is flat.  

The claim \eqref{eqn:claim} follows directly from the proof of \cite[Theorem 1.1]{ChanLee2025}. We outline the steps. By \cite[Lemma 3.2]{ChanLee2025}, $n\geq 4$ and 
\eqref{1_10_i}, $\Scal_{g(t)}$ satisfies $$\left(\partial_t-\Delta_{g(t)} - \Scal_{g(t)}\right) \Scal_{g(t)}(x_0)\leq 0$$ and hence by the maximum principle, thanks to the bounded curvature for $t>0$,
\begin{equation*}
    \Scal_{g(t)}(x_0)\leq \int_M G(x_0,t;y,0) \, \Scal_{g_0} (y)\,d\mathrm{vol}_{g_0}(y)\, ,
\end{equation*}
 where $G(x,t;y,s)$ denotes the heat kernel of the operator $\partial_t-\Delta_{g(t)}-\Scal_{g(t)}$ along the Ricci flow $g(t)$. Owing to the curvature bound \eqref{1_10_ii}, the heat kernel $G(x,t;y,s)$ satisfies the Li-Yau type estimate in \cite[Lemma 3.3]{ChanLee2025}. Hence, the derivation of \cite[(5.14)]{ChanLee2025} can be carried over verbatim to conclude that for any $x_0\in M,\delta>0$, there exists $\Lambda,r_0>0$ so that for all $t>r_0^2$, 
\begin{equation*}
    t \,    \Scal_{g(t)}(x_0) \leq C_1\delta +\frac{C_1 r_0^{n+2}\Lambda t^{-1/2}}{\VolB_{g_0}(x_0,1)}
\end{equation*}
for some dimensional constant $C_1$.
By taking $\liminf_{t\to +\infty}$  followed by $\delta\downto 0$, this proves \eqref{eqn:claim}. 
\end{proof}


\noindent

AD: 
\url{alix.deruelle@universite-paris-saclay.fr}\\
\noindent
{\sc Universit\'e Paris-Saclay, CNRS, Laboratoire de math\'ematiques d’Orsay, 91405, Orsay, France}

\noindent
MCL: 
\url{mclee@math.cuhk.edu.hk}\\
\noindent
{\sc Department of Mathematics, The Chinese University of Hong Kong, Shatin,
N.T., Hong Kong}

\noindent
FS:
\url{felix.schulze@warwick.ac.uk}\\
\noindent
{\sc Mathematics Institute, University of Warwick, Coventry,
CV4 7AL, UK.}

\noindent
MS: 
\url{miles.simon@ovgu.de}\\
\noindent
{\sc Institut f\"ur Analysis und Numerik, Otto-von-Guericke-Universit\"at Magdeburg, 
universit\"atsplatz 2, Magdeburg 39106, Germany}

\noindent
PT: 
\url{http://warwick.ac.uk/fac/sci/maths/people/staff/peter\_topping/}\\
\noindent
{\sc Mathematics Institute, University of Warwick, Coventry,
CV4 7AL, UK.}

\end{document}